\documentclass[preprint,12pt,a4paper]{elsarticle}

\usepackage[T1]{fontenc}
\usepackage{lmodern}
\usepackage{amsmath,amssymb,amsthm,mathtools}
\usepackage{microtype}
\usepackage{etoolbox}
\usepackage{hyperref}

\biboptions{sort&compress}

\hypersetup
{
colorlinks=true,
linkcolor=blue,
filecolor=blue,
urlcolor=blue,
citecolor=cyan,
}

\newtheorem{theorem}{Theorem}[section]
\newtheorem{lemma}[theorem]{Lemma}
\newtheorem{proposition}[theorem]{Proposition}

\DeclareMathOperator{\diam}{diam}

\makeatletter
\renewenvironment{abstract}{%
  \global\setbox\absbox=\vbox\bgroup
  \hsize=\textwidth
  \noindent\hbox to \textwidth{\hfil\textbf{\@elsarticleabstitle}\hfil}%
  \par\medskip\noindent\hspace*{\parindent}\ignorespaces
}{\egroup}
\patchcmd{\pprintMaketitle}{\hrule\vskip12pt}{\vskip12pt}{}{}
\patchcmd{\pprintMaketitle}{\hrule\vskip12pt}{\vskip12pt}{}{}
\makeatother

\begin{document}

\begin{frontmatter}

\title{Clique-saturating non-edges throughout the Tur\'an range}

\author[fzu]{Xiaolin Wang}
\ead{xiaolinw@fzu.edu.cn}

\author[nju]{Jiabao Yang\corref{cor1}}
\ead{jbyang1215@nju.edu.cn}

\author[nju]{Ruilin Zheng}
\ead{rlzhengmath@163.com}

\cortext[cor1]{Corresponding author.}
\address[fzu]{School of Mathematics and Statistics, Fuzhou University, Fuzhou 350108, China}
\address[nju]{School of Mathematics, Nanjing University, Nanjing 210093, China}

\begin{abstract}
For an $F$-free graph $G$, a non-edge is $F$-saturating if adding it to $G$ creates a copy of $F$. We denote by $f_{p+1}(n,m)$ the minimum number of $K_{p+1}$-saturating non-edges in a  $K_{p+1}$-free $n$-vertex graph with $m$ edges.  
Erd\H{o}s and Tuza conjectured that $f_4\left(n,\mathrm{ex}(n,K_3)+ 1\right)= (1 + o(1)) \frac{n^2}{16}$.
Balogh and Liu (JCTB, 2014) disproved this conjecture and determined the asymptotic value of  $f_4(n,\mathrm{ex}(n,K_3)+1)$. He, Ma, Ma and Ye (JCTB, 2023) later determined  $f_{p+1}(n,\mathrm{ex}(n,K_p)+1)$  asymptotically for every $p\ge 3$, and asked for the value of $f_{p+1}(n,m)$ for all $\mathrm{ex}(n,K_p)+1\le m\le \mathrm{ex}(n,K_{p+1})$ and every $p\ge 3$. In this paper, we  answer their question asymptotically for all $\mathrm{ex}(n,K_p)+1\le m\le \mathrm{ex}(n,K_{p+1})$ and  every $p\ge 3$. We also determine the exact value of $f_3(n,m)$ for all $0\le m\le \mathrm{ex}(n,K_3)$ by a different method. 
\end{abstract}


\end{frontmatter}

\section{Introduction}

All graphs considered in this paper are finite and simple. A graph is
$F$-free if it contains no copy of $F$. The classical Tur\'an number
$\mathrm{ex}(n,F)$ denotes the maximum number of edges in an
$n$-vertex $F$-free graph. Mantel~\cite{Mantel} determined
$\mathrm{ex}(n,K_3)$. Tur\'an's theorem~\cite{Turan} states that, for
every integer $p\ge2$, the unique $n$-vertex $K_{p+1}$-free graph
with $\mathrm{ex}(n,K_{p+1})$ edges is the complete balanced
$p$-partite graph $T_p(n)$. We write
\(
t_p(n)=e(T_p(n)).
\)

Let $G$ be an $F$-free graph. A non-edge $xy$ of $G$ is called
$F$-saturating if $G+xy$ contains a copy of $F$. The graph $G$ is
$F$-saturated if every non-edge of $G$ is $F$-saturating. The
saturation number $\operatorname{sat}(n,F)$ is the minimum number of
edges in an $n$-vertex $F$-saturated graph.

We write $f_F(G)$ for the number of $F$-saturating non-edges of $G$.
Furthermore, $f_F(n,m)$ denotes the minimum value of $f_F(G)$ over
all $n$-vertex $F$-free graphs with $m$ edges. When $F=K_r$, we write
$f_r(G)$ and $f_r(n,m)$ in place of $f_{K_r}(G)$ and
$f_{K_r}(n,m)$, respectively.

The parameter $f_F(n,m)$ is closely related to the classical Tur\'an and  saturation problems.  Mantel's theorem~\cite{Mantel} and Tur\'an's theorem~\cite{Turan} determine the largest number of edges in a graph avoiding a prescribed clique, whereas Zykov~\cite{Zykov} and, independently, Erd\H{o}s, Hajnal, and Moon~\cite{EHM} determined the minimum number of edges in an $n$-vertex $K_r$-saturated graph. The function $f_F(n,m)$ refines these extremal questions by fixing the number of edges and asking how many non-edges are already forced to complete a copy of $F$.

In this paper, we study the value of $f_{p+1}(n,m)$ for $p\ge 2$, that is,   $F$ is a clique.
For $p\ge3$, it is easy to check that every spanning subgraph of $T_{p-1}(n)$ has no $K_{p+1}$-saturating non-edge. Hence $f_{p+1}(n,m)=0$ whenever $m\le t_{p-1}(n)$. A surprising result of Erd\H{o}s and Tuza 
\cite{ErdosProblems} showed that 
$f_4(n,t_2(n)+1)\ge cn^2$ for some absolute constant $c>0$. Thus increasing the edge count from \(t_2(n)\) to
\(t_2(n)+1\) causes \(f_4(n,m)\) to jump from \(0\) to
\(\Omega(n^2)\).
Erd\H{o}s and Tuza 
\cite{ErdosProblems} also conjectured that
\(f_4(n,t_2(n)+1)=\left(\frac1{16}+o(1)\right)n^2.
\)

In 2014, Balogh and Liu~\cite{BaloghLiu} disproved this conjecture and established the following theorem.

\begin{theorem}\label{BL}(Balogh and Liu~\cite{BaloghLiu})
\(
f_4(n,t_2(n)+1)=\left(\frac2{33}+o(1)\right)n^2.
\)
\end{theorem}

More precisely, they proved that $f_4(n,t_2(n)+t)=2n^2/33+O(n)$ for every integer $1\le t\le n/66$. They also conjectured that $f_{p+1}(n,t_{p-1}(n)+1)=\left(\frac{2(p-2)^2}{p(4p^2-11p+8)}+o(1)\right)n^2$ for every $p\ge 3$, which was later confirmed by 
He, Ma, Ma, and Ye~\cite{HMMY}.

\begin{theorem}\label{HMMY}(He, Ma, Ma, and Ye~\cite{HMMY}) For every $p\ge 3$,
$f_{p+1}(n,t_{p-1}(n)+1)=\left(\frac{2(p-2)^2}{p(4p^2-11p+8)}+o(1)\right)n^2.$
\end{theorem}

He, Ma, Ma, and Ye~\cite{HMMY} also  asked for the value of $f_{p+1}(n,m)$ throughout the interval between $t_{p-1}(n)+1$ and $t_p(n)$. Our first main theorem answers this question.
Fix $p\ge3$. For $t_{p-1}(n)<m\le t_p(n)$, define
\begin{equation}\label{eq:lambda}
\lambda=\frac{m-t_{p-1}(n)}{t_p(n)-t_{p-1}(n)}.
\end{equation}
Then $0<\lambda\le1$. Let $D_p=4p^2-11p+8$, and let $x_p(\lambda)$ be the unique positive root of
\begin{equation}\label{eq:quartic}
p^3D_px^4-2p^2(p-2)^2x^3-2p(p-2)\lambda x-\lambda^2=0.
\end{equation}
Define
\begin{equation}\label{eq:profile}
\begin{aligned}
y_p(\lambda)&=\frac{\lambda}{2p(p-1)x_p(\lambda)}+\frac{p-2}{p-1}\left(1-\frac{x_p(\lambda)}2\right),\\
\Phi_p(\lambda)&=\frac12\left(x_p(\lambda)^2+\frac{y_p(\lambda)^2}{p-1}\right).
\end{aligned}
\end{equation}

\begin{theorem}\label{thm:clique-main}
For every fixed $p\ge3$, $f_{p+1}(n,m)=\Phi_p(\lambda)n^2+O_p(n)$ for all  $t_{p-1}(n)<m\le t_p(n)$.
\end{theorem}

For the upper bound in Theorem~\ref{thm:clique-main}, we use a family of blow-ups extending the constructions in~\cite{BaloghLiu,HMMY}. For the lower bound, we choose a maximum collection $\mathcal R$ of vertex-disjoint copies of $K_p$ whose deletion leaves as many edges as possible. A switching argument produces a distinguished $K_p$ for which one relevant common-neighborhood class is empty. 
We divide the saturating non-edges into those incident with $V(\mathcal R)$ and those contained in $V(G)\setminus V(\mathcal R)$. Estimating these two types reduces the proof to a one-variable optimization.

The one-sided endpoint values of the profile agree with the known extremal cases. At $\lambda=0$, a simple calculation shows that 
\[
x_p(0)=\frac{2(p-2)^2}{p(4p^2-11p+8)}\quad\text{and}\quad \Phi_p(0)=x_p(0),
\]
so Theorem \ref{thm:clique-main} recovers the leading term in Theorem \ref{HMMY} whenever $m-t_{p-1}(n)$ is positive and equals $o(n^2)$. Here $\Phi_p(0)$ is the right-hand limit above the threshold; at $m=t_{p-1}(n)$ itself, the minimum equals zero. At $\lambda=1$, one has $x_p(1)=1/p$, $y_p(1)=(p-1)/p$, and $\Phi_p(1)=1/(2p)$. Hence
\[
f_{p+1}(n,t_p(n))=\binom n2-t_p(n)=\frac{n^2}{2p}+O_p(n).
\]

 For $f_3(n,m)$,  since $t_1(n)+1=1$, it is easy to see that $f_3(n,t_1(n)+1)=0$. So $f_3(n,m)$ must behave quite differently from  $f_{p+1}(n,m)$ for $p\ge 3$. Our second result determines the exact value of $f_3(n,m)$.   In a triangle-free graph, a non-edge is triangle-saturating exactly when its endpoints have a common neighbor. Thus $f_3(G)$ is the number of unordered pairs of vertices at distance two. Furthermore,  when $m\le\lfloor n/2\rfloor$, there is an $n$-vertex graph with $m$ edges whose edge set is a matching. Then  $f_3(n,m)=0$. When $m>\lfloor n/2\rfloor$, at least one non-edge is $K_3$-saturating in each $n$-vertex triangle-free graph with $m$ edges. Hence $f_3(n,m)=0$ exactly when $m\le\lfloor n/2\rfloor$.

 For every positive integer $q$, let $t(q)=\lfloor q^2/4\rfloor$ and $s(q)=\lfloor(q-1)^2/4\rfloor$. It is easy to calculate that $T_2(q)$ has  exactly $s(q)$ triangle-saturating non-edges.

\begin{theorem}\label{thm:triangle-main}
For all integers $n\ge1$ and $0\le m\le t(n)$,
\begin{equation}\label{eq:triangle-formula}
f_3(n,m)=\min\left\{\sum_{i=1}^k s(q_i): k,q_1,\ldots,q_k\in\mathbb N,\ \sum_{i=1}^k q_i=n,\ \sum_{i=1}^k t(q_i)\ge m\right\}.
\end{equation}
\end{theorem}

The upper bound in Theorem~\ref{thm:triangle-main} is obtained
from a disjoint union of balanced complete bipartite graphs,
followed by deleting edges if necessary, while the lower bound follows from a componentwise estimate for connected triangle-free graphs.

In particular, \(f_3(n,t(n))=s(n)\). Indeed, by the equality
case of Mantel's theorem, every \(n\)-vertex triangle-free graph
with \(t(n)\) edges is isomorphic to \(T_2(n)\). Its
triangle-saturating non-edges are precisely the pairs contained
in one of its two parts, and their number is \(s(n)\).

Since the right-hand side of \eqref{eq:triangle-formula} is the value of a finite optimization problem, it can be computed by the following
dynamic program. Let $F(n,m)$ denote the right-hand side of \eqref{eq:triangle-formula}. For nonnegative integers $n,m$, we extend $F(n,m)$ by allowing the empty partition when $n=0$ with $F(0,0)=0$, and use $F(n,m)=+\infty$ when no admissible partition exists.

\begin{theorem}\label{cor:triangle-dp}
For $n\ge1$ and $0\le m\le t(n)$,
\begin{equation}\label{eq:triangle-dp}
F(n,m)=\min_{1\le q\le n}\left\{s(q)+F\bigl(n-q,\max\{0,m-t(q)\}\bigr)\right\}.
\end{equation}
\end{theorem}

\begin{proof} 
In an admissible partition of $n$, distinguish a part of size $q$. The remaining parts have total order $n-q$ and total $t$-value at least $\max\{0,m-t(q)\}$. This proves that the left-hand side of \eqref{eq:triangle-dp} is at least the right-hand side. Conversely, adjoin a part of size $q$ to a partition attaining the corresponding value of $F(n-q,\max\{0,m-t(q)\})$. The resulting partition is admissible for $F(n,m)$, which proves the reverse inequality.
\end{proof}

At the end of this section, we use the following notation throughout the paper. For a vertex $v$ in a graph $G$, let $N_G(v)$ and $d_G(v)$ denote its neighborhood and degree, respectively. 
For $X\subseteq V(G)$, let $G[X]$ be the subgraph induced by $X$ and let $G-X=G[V(G)\setminus X]$. For disjoint vertex sets $A$ and $B$, let $e(A,B)$ be the number of edges between them. We use the same notation for disjoint subgraphs, identifying each subgraph with its vertex set.

The rest of this paper is organized as follows. 
Sections~\ref{sec:2} and \ref{sec:3} prove the upper bound and the lower bound of Theorem \ref{thm:clique-main}, respectively.
In Section~\ref{sec:4}, we prove  Theorem~\ref{thm:triangle-main}.

\section{The upper bound in Theorem~\ref{thm:clique-main}}\label{sec:2}

Throughout Sections~\ref{sec:2} and~\ref{sec:3}, fix an integer
$p\ge3$. Let $x=x_p(\lambda)$, $y=y_p(\lambda)$, and $\Phi_p(\lambda)=\Phi$ be defined by \eqref{eq:quartic} and \eqref{eq:profile}.

\begin{lemma}\label{lem:profile-parameters}
For every $\lambda\in[0,1]$, one has $\sqrt{\lambda}/p\le x\le 1/p$ and $0<y\le1-x$. Moreover, $x,y,\Phi$ are continuously differentiable on $[0,1]$.
\end{lemma}

\begin{proof}
Let $$Q_\lambda(X)=p^3D_pX^4-2p^2(p-2)^2X^3-2p(p-2)\lambda X-\lambda^2$$ be the left-hand side of \eqref{eq:quartic}. When $\lambda=0$, its unique positive root $x$ is $2(p-2)^2/(pD_p)$, and we have $\sqrt{\lambda}/p=0< x\le 1/p$. When $\lambda>0$, since $D_p=(p-2)(4p-3)+2>0$, its nonzero coefficients have one sign change. By descartes’ rule of signs, $Q_\lambda(0)<0$, and $Q_\lambda(X)>0$ for all sufficiently large $X$, we obtain $Q_\lambda$ has exactly one positive root. Since $Q_\lambda(1/p)=(1-\lambda)(\lambda+2p-3)\ge0$ and $Q_\lambda\!\left(\frac{\sqrt{\lambda}}{p}\right)
=\frac{4(p-1)(p-2)}{p}\,\lambda^2\left(1-\frac1{\sqrt{\lambda}}\right)\le0$,
we have $\sqrt{\lambda}/p\le x\le1/p$ as well.

Since $x>0$, $\lambda\ge0$, and $p\ge3$, the definition of $y$ implies that $y>0$.
By \eqref{eq:profile}, we have $$x+y=\frac{p-2}{p-1}+\frac{px}{2(p-1)}+\frac{\lambda}{2p(p-1)x}.$$ 
Let $f(X)=\frac{p-2}{p-1}+\frac{pX}{2(p-1)}+\frac{\lambda}{2p(p-1)X}$. 
For $X\ge\sqrt{\lambda}/p$, we have 
$$f'(X)=\frac{p}{2(p-1)}-\frac{\lambda}{2p(p-1)X^2}\ge0,$$ 
and hence
$$x+y=f(x)\le f(1/p)
=\frac{p-2}{p-1}+\frac{1}{2(p-1)}+\frac{\lambda}{2(p-1)}
=\frac{2p-3+\lambda}{2(p-1)}\le1.$$

Finally, from  \eqref{eq:quartic},
$xQ'_\lambda(x)=2p^2(p-2)^2x^3+6p(p-2)\lambda x+4\lambda^2>0$,
so the root is simple for every $\lambda\in[0,1]$. The implicit function theorem shows that $x$ is continuously differentiable. The same conclusion then follows for $y$ and $\Phi$.
\end{proof}

Consider a graph $G$ whose vertex classes are
$V_0,V_1,\ldots,V_{p-1},U_1,\ldots,U_{p-1}$.
For distinct $i,j\in\{1,\ldots,p-1\}$, include every edge between $V_i\cup U_i$ and $V_j\cup U_j$, and include every edge between $V_0$ and $V_i$. There are no other edges.

\begin{lemma}\label{lem:blowup}
The graph above is $K_{p+1}$-free. If $V_0,V_1,\ldots,V_{p-1}$ are nonempty, its $K_{p+1}$-saturating non-edges are exactly the pairs within $V_0$ and the pairs within the classes $V_i$. If
$|V_0|=xn+O_p(1)$, $|V_i|=yn/(p-1)+O_p(1)$, and $|U_i|=(1-x-y)n/(p-1)+O_p(1)$, then
\begin{align}
 e(G)&=\left(\frac{p-2}{2(p-1)}(1-x)^2+xy\right)n^2+O_p(n),\label{eq:blowup-edges}\\
 f_{p+1}(G)&=\frac12\left(x^2+\frac{y^2}{p-1}\right)n^2+O_p(n).\label{eq:blowup-saturation}
\end{align}
\end{lemma}

\begin{proof}
A clique avoiding $V_0$ contains at most one vertex from each class $V_i\cup U_i$, so it has at most $p-1$ vertices. A clique meeting $V_0$ avoids every $U_i$ and contains at most one vertex from each $V_i$, so it has at most $p$ vertices. Thus the graph is $K_{p+1}$-free.

A pair within $V_0$ has a common $K_{p-1}$ formed by one vertex from each $V_i$. A pair within $V_i$ has a common $K_{p-1}$ formed by one vertex from $V_0$ and one vertex from every $V_j$ with $j\ne i$. Hence these pairs are saturating.

For a pair within $U_i$, or a pair with one endpoint in $V_i$ and one in $U_i$, the common neighborhood meets only the $p-2$ classes $V_j\cup U_j$ with $j\ne i$, so its clique number is at most $p-2$. For a non-edge between $V_0$ and $U_i$, the common neighborhood is contained in the union of the $p-2$ classes $V_j$ with $j\ne i$. These pairs are not saturating, and all non-edges have now been considered.

The edges between distinct classes $V_i\cup U_i$ contribute $\frac{p-2}{2(p-1)}(1-x)^2n^2+O_p(n)$, and the edges from $V_0$ to the classes $V_i$ contribute $xyn^2+O_p(n)$. This proves \eqref{eq:blowup-edges}. The classification of saturating non-edges yields 
\[
\binom{|V_0|}{2}+\sum_{i=1}^{p-1}\binom{|V_i|}{2}=\frac12\left(x^2+\frac{y^2}{p-1}\right)n^2+O_p(n),
\]
which proves \eqref{eq:blowup-saturation}.
\end{proof}

For $x=x_p(\lambda)$ and $y=y_p(\lambda)$, the definitions in
\eqref{eq:profile} yield the following expression for the edge
density in \eqref{eq:blowup-edges}:
\begin{equation}\label{eq:edge-density}
\frac{p-2}{2(p-1)}(1-x)^2+xy
=
\frac{p-2}{2(p-1)}+\frac{\lambda}{2p(p-1)}.
\end{equation}

It is worth mentioning that if $n\equiv s\pmod r$ and $0\le s<r$, then $t_r(n)=\frac{r-1}{2r}n^2-\frac{s(r-s)}{2r}$. Consequently,
\begin{equation}\label{eq:turan-difference}
t_p(n)-t_{p-1}(n)=\frac{n^2}{2p(p-1)}+O_p(1).
\end{equation}

\begin{proposition}\label{prop:clique-upper}
For every fixed $p\ge3$ and every $t_{p-1}(n)<m\le t_p(n)$, we have $f_{p+1}(n,m)\le\Phi_p(\lambda)n^2+O_p(n)$.
\end{proposition}

\begin{proof}
Choose integer class sizes summing to $n$ such that each differs by $O_p(1)$ from the corresponding real quantity in Lemma~\ref{lem:blowup}. Since $x$ and $y$ are positive and continuous on $[0,1]$, each has a positive minimum on this interval. Thus the prescribed sizes of $V_0,V_1,\ldots,V_{p-1}$ are all $\Omega_p(n)$, and hence their rounded sizes are positive for all sufficiently large $n$. The formulas in Lemma~\ref{lem:blowup} express both the number of edges and the number of saturating non-edges as sums of $O_p(1)$ quadratic terms in the class sizes. Therefore, changing each class size by $O_p(1)$ changes either count by at most $O_p(n)$. Finally, Lemma~\ref{lem:profile-parameters} shows that $x$, $y$, and $\Phi$ are Lipschitz on $[0,1]$.


Let $C_0=C_0(p)$ satisfy that rounding changes the edge number by at most $C_0n$. Choose $C>2p(p-1)(C_0+1)$. 

Suppose first that $\lambda\le1-C/n$, and use the construction with $\lambda'=\lambda+C/n$. By \eqref{eq:edge-density} and \eqref{eq:turan-difference} , before rounding its edge number exceeds $m$ by $Cn/[2p(p-1)]+O_p(1)$. Thus the rounded graph still has at least $m$ edges. Delete edges until exactly $m$ remain.

If $H$ is a spanning subgraph of a $K_{p+1}$-free graph $G$, every $K_{p+1}$-saturating non-edge of $H$ is also saturating in $G$. Indeed, such a pair cannot be an edge of $G$, because then the created $K_{p+1}$ would already lie in $G$. Hence deleting edges cannot increase the number of saturating non-edges. Lemma~\ref{lem:blowup} and the Lipschitz property now yield $f_{p+1}(n,m)\le\Phi_p(\lambda')n^2+O_p(n)=\Phi_p(\lambda)n^2+O_p(n)$.

Now suppose that $\lambda>1-C/n$. Begin with $T_p(n)$ and delete edges until $m$ remain. Every non-edge of $T_p(n)$ is $K_{p+1}$-saturating. Since $\Phi_p(1)=1/(2p)$ and $|1-\lambda|=O_p(1/n)$, the Lipschitz property yields $f_{p+1}(n,m)\le\binom n2-t_p(n)=\Phi_p(\lambda)n^2+O_p(n)$.

The proof is complete.
\end{proof}

\section{The lower bound in Theorem~\ref{thm:clique-main}}\label{sec:3}

Let $G$ be an $n$-vertex $K_{p+1}$-free graph with $m>t_{p-1}(n)$ edges. For convenience, we write 
$$a=\frac{p-2}{p-1}, \quad  \eta=\frac{m-t_{p-1}(n)}{n^2}, \quad \text{and} \quad \alpha=\frac{\lambda}{2p(p-1)}.$$ 
By \eqref{eq:turan-difference}, $\eta=\alpha+O_p(n^{-2})$, while $e(G)=(a/2+\eta)n^2+O_p(1)$.

By Tur\'an's theorem, $G$ contains a copy of $K_p$. A $K_p$-packing
in $G$ is a family of pairwise vertex-disjoint copies of $K_p$.
Choose a $K_p$-packing $\mathcal R$ of maximum cardinality. Subject
to this condition, choose $\mathcal R$ so that
$e(G-V(\mathcal R))$ is as large as possible. Let
$V(\mathcal R)=\bigcup_{R\in\mathcal R}V(R)
$ and $H=G-V(\mathcal R)$.
Then $H$ is $K_p$-free.

For $R=\{v_1,\ldots,v_p\}\in\mathcal R$ and $1\le i\le p$, let $A_i(R)$ consist of the vertices $u\in V(H)$ adjacent to every vertex of $R\setminus\{v_i\}$. 
In the following, we sometimes omit the specific set of vertices in $R$ and still use the formal notation $A_i(R)$.
These $p$ classes (that is $A_1(R),\cdots, A_p(R)$) are pairwise disjoint and independent. Indeed, a vertex belonging to two distinct classes would be adjacent to every vertex of $R$, and an edge inside one class, together with $R$ minus the corresponding vertex, would form a copy of $K_{p+1}$. Every pair within one class is $K_{p+1}$-saturating, because $R\setminus\{v_i\}$ is a copy of $K_{p-1}$ in the common neighborhood of its two endpoints.

\begin{lemma}\label{lem:switching}
Let $R\in\mathcal R$, let $X\subseteq R$ and $Y\subseteq V(H)$ be cliques with $|X|=|Y|$, and assume that $R'=(R\setminus X)\cup Y$ is a copy of $K_p$. Replace $R$ by $R'$ in $\mathcal R$, and let $H'$ be the new remaining graph. Then $e(R',H')\ge e(R,H)$.
\end{lemma}

\begin{proof}
The new family is a maximum $K_p$-packing and $H'=(H\setminus Y)\cup X$. The choice of $\mathcal R$ implies $e(H)\ge e(H')$. Since $X$ and $Y$ are cliques of the same order, this inequality is equivalent to $e(Y,H\setminus Y)\ge e(X,H\setminus Y)$. The edges from $R\setminus X$ to $X$ and to $Y$ are complete, while every other term cancels. Therefore $e(R',H')-e(R,H)=e(Y,H\setminus Y)-e(X,H\setminus Y)\ge0$.
\end{proof}

\begin{lemma}\label{lem:empty-class}
Suppose that $R^*\in\mathcal R$ satisfies $e(R^*,H)\ge L$. There is a maximum $K_p$-packing $\mathcal Q$ and a clique $Q^*\in\mathcal Q$ such that $e(Q^*,G-V(\mathcal Q))\ge L$ and one of the classes $A_1(Q^*),\cdots, A_p(Q^*)$ is empty.
\end{lemma}

\begin{proof}
Suppose, to the contrary, that in every maximum $K_p$-packing, each clique with at least $L$ edges to the remaining graph has all $p$ classes nonempty. 
Since every $A_i(R^*)$ is nonempty, a vertex of any one class forms a singleton clique $C\subseteq H$ such that $R^*\cup C$ contains a copy of $K_p$ containing every vertex of $C$. Hence there is a nonempty clique $C\subseteq H$ with this property. Choose such a clique $C$ of maximum order. Since $H$ is $K_p$-free, write $C=\{x_1,\ldots,x_c\}$, where $1\le c\le p-1$. Relabel $R^*=\{v_1,\ldots,v_p\}$ so that $R'=\{x_1,\ldots,x_c,v_{c+1},\ldots,v_p\}$ is a copy of $K_p$.

Let $\mathcal Q=(\mathcal R\setminus\{R^*\})\cup\{R'\}$ and $H'=G-V(\mathcal Q)$. Then $\mathcal Q$ is a maximum $K_p$-packing. By Lemma~\ref{lem:switching}, $e(R',H')\ge L$. Applying the contrary assumption to the pair $(\mathcal Q,R')$, we conclude that every class $A_i(R')$, now defined with respect to $H'$, is nonempty.

Choose $y\in A_p(R')$. Since $y$ is adjacent to every vertex of $R'\setminus\{v_p\}$ and $G$ is $K_{p+1}$-free, it is not adjacent to $v_p$. On the other hand, each of $v_1,\ldots,v_c$ is adjacent to $v_p$, so $y\notin\{v_1,\ldots,v_c\}$. Since $V(H')=(V(H)\setminus C)\cup\{v_1,\ldots,v_c\},$
it follows that $y\in V(H)\setminus C$.

The vertex $y$ is adjacent to every vertex of $C$ and to $v_{c+1},\ldots,v_{p-1}$. Hence $C\cup\{y\}$ is a clique in $H$, while $C\cup\{y,v_{c+1},\ldots,v_{p-1}\}$ is a copy of $K_p$ containing every vertex of $C\cup\{y\}$. This contradicts the maximality of $C$.
\end{proof}

For a real number $w$, let $w_+=\max\{w,0\}$. Define $b=p(2p-3)/[2(p-1)]$ and, for $0<r\le1/p$,
\begin{equation}\label{eq:Falpha}
F_\alpha(r)=\alpha+ar-\frac{pa}{2}r^2+\frac1{2(p-1)}\left(\frac\alpha r+a-br\right)_+^2.
\end{equation}

\begin{proposition}\label{prop:clique-lower-function}
Every $n$-vertex $K_{p+1}$-free graph $G$ with $m>t_{p-1}(n)$ edges satisfies $f_{p+1}(G)\ge F_\alpha(|\mathcal R|/n)n^2-O_p(n)$.
\end{proposition}

\begin{proof}
For convenience, we write $s=|\mathcal R|$ and $r=s/n$. 
Since $H$ is $K_p$-free, Tur\'an's theorem yields $e(H)\le a(1-pr)^2n^2/2$. The induced graph on $V(\mathcal R)$ is $K_{p+1}$-free, so $e(G[V(\mathcal R)])\le p(p-1)r^2n^2/2$. Subtracting these bounds from $e(G)$ yields
\begin{equation}\label{eq:cross-lower}
e(V(\mathcal R),H)\ge\left(\eta+\frac{p(p-2)}{p-1}r-\frac{p(2p^2-4p+1)}{2(p-1)}r^2\right)n^2-O_p(1).
\end{equation}
Averaging \eqref{eq:cross-lower} over the $s=rn$ cliques and then applying Lemma~\ref{lem:empty-class}, we may replace the $K_p$-packing and relabel its remaining graph so that a clique $R^*$ has one empty class $A_i(R^*)$ for some $i\in\{1,\ldots,p\}$ and satisfies
\begin{equation}\label{eq:distinguished-lower}
e(R^*,H)\ge\left(\frac\eta r+\frac{p(p-2)}{p-1}-\frac{p(2p^2-4p+1)}{2(p-1)}r\right)n-O_p(1).
\end{equation}
Let $zn$ be the total number of vertices in the $A_1(R^*)\cup \cdots \cup A_p(R^*)$. Every vertex of $H$ has at most $p-1$ neighbors in $R^*$; vertices in these $p$ classes have exactly $p-1$, and every other vertex has at most $p-2$. Hence $e(R^*,H)\le(p-2)|H|+zn$. Combining this inequality with \eqref{eq:distinguished-lower}, $|H|=(1-pr)n$, and $\eta=\alpha+O_p(n^{-2})$, we obtain
\begin{equation}\label{eq:z-lower}
z\ge\frac\alpha r+a-br-O_p(n^{-1}).
\end{equation}

Let $\ell_1$ be the number of saturating non-edges incident with $V(\mathcal R)$. Order the cliques as $R_1,\ldots,R_s$, and let $d_i$ be the number of edges from $R_i$ to the vertices outside $R_1\cup\cdots\cup R_i$. Every such remaining vertex has at most $p-1$ neighbors in $R_i$, so at least $d_i-(p-2)(n-pi)$ of them have exactly $p-1$ neighbors in $R_i$. Each of these vertices forms a saturating non-edge with its unique non-neighbor in $R_i$. No pair is counted twice, because its endpoint in the first processed clique determines $i$. Therefore
\begin{equation}\label{eq:incident-lower}
\begin{aligned}
\ell_1&\ge\sum_{i=1}^s\bigl(d_i-(p-2)(n-pi)\bigr)\\
&=e(G)-e(H)-\binom p2s-(p-2)\left(sn-\frac p2s(s+1)\right)\\
&\ge\left(\alpha+ar-\frac{pa}{2}r^2\right)n^2-O_p(n).
\end{aligned}
\end{equation}

Let $\ell_2$ be the number of saturating non-edges within $H$. Since one class $A_i(R^*)$ is empty, the $zn$ vertices counted above lie in at most $p-1$ independent classes, and every pair within one class is saturating. Convexity yields $\ell_2\ge z^2n^2/[2(p-1)]-zn/2$. 
It follows from $0\le z\le1$ and \eqref{eq:z-lower} that $z^2\ge (\alpha/r+a-br)_+^2-O_p(n^{-1})$.
Hence
\begin{equation}\label{eq:remainder-lower}
\ell_2\ge\frac1{2(p-1)}\left(\frac\alpha r+a-br\right)_+^2n^2-O_p(n).
\end{equation}

The pairs counted by $\ell_1$ and $\ell_2$ are disjoint. Combining \eqref{eq:incident-lower} and \eqref{eq:remainder-lower} proves the proposition.
\end{proof}

\begin{lemma}\label{lem:optimization}
$F_\alpha(r)\ge\Phi_p(\lambda)$ for every $0<r\le1/p$.
\end{lemma}
\begin{proof}
Since $m>t_{p-1}(n)$, we have $\lambda>0$ and $\alpha>0$. Recall that $b=p(2p-3)/[2(p-1)]$.
Let $g(r)=\alpha/r+a-br$. Then $F_\alpha(r)=\alpha+ar-(pa/2)r^2+\frac1{2(p-1)}g(r)_+^2$. Since $g'(r)=-\alpha/r^2-b<0$, $g$ is strictly decreasing. Since $g(1/p)=(\lambda-1)/[2(p-1)]\le0$ and $\lim_{r\to0^+}g(r)=+\infty$, there is a unique $\rho\in(0,1/p]$ with $g(\rho)=0$, and $g(r)\ge0$ for $0<r\le\rho$, $g(r)\le0$ for $\rho\le r\le1/p$.
Therefore 
\begin{align*}
F_\alpha(r)=  \alpha+ar-\frac{pa}{2}r^2+
\begin{cases}
\frac{g(r)^2}{2(p-1)} & \text{for } r \in (0,\rho],\\
0 & \text{for } r \in [\rho,1/p].
\end{cases}
\end{align*}

We now estimate the minimum value of $F_\alpha(r)$ on the interval $(0, 1/p]$.

If $r \in (0,\rho]$, let $S(r):=F_\alpha(r)=\alpha+ar-(pa/2)r^2+\frac1{2(p-1)}g(r)^2$. 
Then $$S'(r)=a-par+\frac1{p-1}(-\alpha/r^2-b)g(r)=\frac{P_\alpha(r)}{4r^3(p-1)^3},$$ 
where $$P_\alpha(r)=pD_pr^4-2(p-2)^2r^3-4(p-1)(p-2)\alpha r-4(p-1)^2\alpha^2.$$ 
After substituting $\alpha=\lambda/[2p(p-1)]$, we deduce that $P_\alpha(r)=Q_{\lambda}(r)/p^2$, where $Q_{\lambda}(X)$ is as defined in the proof of Lemma \ref{lem:profile-parameters}.
As in the proof of Lemma \ref{lem:profile-parameters}, $P_\alpha$ has a unique positive root $x=x_p(\lambda)$ with $x\le1/p$, and $P_\alpha<0$ on $(0,x)$, $P_\alpha>0$ on $(x,\infty)$. 
Hence $S$ decreases on $(0,x)$ and increases on $(x,\infty)$. 
Using $g(\rho)=0$ in $S'(\rho)$ yields $S'(\rho)=a(1-p\rho)\ge0$, so $x\le\rho$. Therefore $\min_{(0,\rho]}S=S(x)$. 
If $r\in [\rho,1/p]$, let $T(r):=F_\alpha(r)=\alpha+ar-(pa/2)r^2$.
Then $T'(r)=a(1-pr)\ge0$, so $\min_{\rho \le r\le1/p} T=T(\rho)$. 
Hence $\min_{0<r\le1/p}F_\alpha(r)=\min \{S(x),T(\rho)\}=S(x)$ by $T(\rho)=S(\rho)$ and $\min_{(0,\rho]}S=S(x)$.

Let $y=\frac{\alpha}{x}+a-\frac a2x$. Since \(b=(p-1)+a/2\), we have
$g(x)=y-(p-1)x$. A direct simplification gives $S(x)=\frac12\left(x_p(\lambda)^2+\frac{y_p(\lambda)^2}{p-1}\right)=\Phi_p(\lambda)$.
 
Hence, for every \(0<r\le1/p\),
$F_\alpha(r)\ge S(x)=\Phi_p(\lambda)$ as desired.
\end{proof}

Note that $1/n\le |\mathcal R|/n\le \lfloor n/p \rfloor/n \le1/p$.
For every admissible graph $G$, Proposition~\ref{prop:clique-lower-function} and Lemma~\ref{lem:optimization} yield $f_{p+1}(G)\ge\Phi_p(\lambda)n^2-O_p(n)$. Minimizing over $G$ proves the lower bound of Theorem \ref{thm:clique-main}.


\section{Proof of Theorem~\ref{thm:triangle-main}}\label{sec:4}

 Let $N_2(v)$ be the set of vertices at distance two from $v$. Let $u(v)$ be the number of unordered pairs of vertices in $N_G(v)$ whose unique common neighbor is $v$. We first prove Theorem~\ref{thm:triangle-main} for the connected case, and then prove Theorem~\ref{thm:triangle-main} completely.

\begin{lemma}\label{lem:triangle-identities}
Let $G$ be triangle-free.
\begin{enumerate}
\item[(1)] If $H=G-v$ for some $v\in V(G)$, then
$f_3(G)-f_3(H)=|N_2(v)|+u(v).$

\item[(2)] Suppose that $G_1$ and $G_2$ are connected subgraphs of $G$,  $G_1\cap G_2=\{x\}$ and $G=G_1\cup G_2$. Then
$f_3(G)=f_3(G_1)+f_3(G_2)+d_{G_1}(x)d_{G_2}(x).
$
\end{enumerate}
\end{lemma}

\begin{proof}
Since every triangle-saturating non-edge of $H$ is also a  triangle-saturating  non-edge in $G$, we only need to consider the triangle-saturating non-edges in $G$ but not in $H$. A pair counted by $f_3(G)$ but not by $f_3(H)$ either contains $v$, in which case its other endpoint belongs to $N_2(v)$, or has both endpoints in $N_G(v)$ and loses its unique common neighbor when $v$ is deleted. Hence the first statement holds.

For the second statement, a pair contained in $V(G_i)$ is triangle-saturating in $G$ exactly when it is triangle-saturating in $G_i$. A pair with one endpoint in each $G_i$ has a common neighbor exactly when both endpoints are adjacent to $x$. And there are $d_{G_1}(x)d_{G_2}(x)$ such pairs. Combining the above counting, we see the second statement holds.
\end{proof}

For a triangle-free graph $G$ and $v\in V(G)$, by Lemma \ref{lem:triangle-identities}(1), let $\ell(v)=f_3(G)-f_3(G-v)=|N_2(v)|+u(v)$.

\begin{lemma}\label{lem:triangle-removal}
Let $G$ be a $2$-connected triangle-free graph with minimum degree $\delta$. If $G\ne K_{\delta,\delta}$, then there is a vertex $v$ such that $G-v$ is connected and $\ell(v)\ge d_G(v)$.
\end{lemma}

\begin{proof}
Choose a vertex $v$ of degree $\delta$. Note that  $G-v$ is connected and  every neighbor of $v$ has at least $\delta-1$ further neighbors, all in $N_2(v)$. Then  $|N_2(v)|\ge\delta-1$. If $|N_2(v)|\ge\delta$, then $\ell(v)\ge\delta=d_G(v)$.

Assume that $|N_2(v)|=\delta-1$. For every $y\in N_G(v)$, the set $N_G(y)\setminus\{v\}$ is contained in $N_2(v)$ and has at least $\delta-1$ vertices. Hence $d_G(y)=\delta$ and $N_G(y)\setminus\{v\}=N_2(v)$. That is, every vertex of $N_G(v)$ is adjacent to every vertex of $N_2(v)$. Indeed, since $G$ is triangle-free,  $N_G(v)$ and $N_2(v)$ are both independent. Hence the subgraph induced by $N_G(v)\cup\{v\}\cup N_2(v)$ is a copy of $K_{\delta,\delta}$, with parts $N_G(v)$ and $\{v\}\cup N_2(v)$.

Since $G\ne K_{\delta,\delta}$ and $G$ is connected, there is an edge $xz$ with $x$ in this complete bipartite subgraph and $z$ outside it. Since no vertex of $N_G(v)\cup \{v\}$ has an outside neighbor,  $x\in N_2(v)$. Choose $y\in N_G(v)$. The $\delta-1$ vertices of $N_G(v)\setminus\{y\}$ are at distance two from $y$ through $v$, and $z$ is at distance two from $y$ through $x$. Thus $|N_2(y)|\ge\delta=d_G(y)$. Since $G$ is $2$-connected, $G-y$ is connected, and $y$ has the required properties.
\end{proof}

\begin{lemma}\label{lem:triangle-small-diameter}
Let $H$ be a connected bipartite graph of diameter at most three, with bipartition $A\cup B$, where $|A|=a$ and $|B|=b$. Then
\begin{equation}\label{eq:diameter-count}
f_3(H)=\binom a2+\binom b2.
\end{equation}
Moreover, for every vertex $x\in V(H)$,  $f_3(H)\ge e(H)-d_H(x)$ and $f_3(H)\ge s(a+b)$.
\end{lemma}

\begin{proof}
Since any two vertices in the same part have even distance,  and since the  diameter is at most 3, its distance must be two. Also,  vertices in different parts cannot have a common neighbor. These arguments prove \eqref{eq:diameter-count}.

Assume that $x\in A$. Every edge not incident with $x$ joins $A\setminus\{x\}$ to $B$, so $e(H)-d_H(x)\le(a-1)b$. By \eqref{eq:diameter-count}, $f_3(H)-(a-1)b=(a-b)(a-b-1)/2\ge0$ because $a-b$ is an integer. The case $x\in B$ is symmetric.

Finally, if $q=a+b$, then $f_3(H)=(q^2-2q+(a-b)^2)/4$. When $q$ is even, $a-b$ is even; when $q$ is odd, $a-b$ is odd. Hence $(a-b)^2\ge0$ in the first case and $(a-b)^2\ge1$ in the second, which yields $f_3(H)\ge\lfloor(q-1)^2/4\rfloor=s(q)$.
\end{proof}

\begin{lemma}\label{thm:triangle-structural}
Let $G$ be a connected triangle-free graph with $e=e(G)$.
\begin{enumerate}
\item[(1)] If $G$ is not bipartite, then $f_3(G)\ge e$.
\item[(2)] If $G$ is bipartite and $\diam(G)\ge4$, then $f_3(G)\ge e-1$.
\end{enumerate}
\end{lemma}

\begin{proof}
We use  induction on $|V(G)|$. Obviously, both assertions are vacuous when $|V(G)|\le4$. If $|V(G)|=5$, it is easy to check that $G$ is a 5-cycle when $G$ is not bipartite and (1) is true, or $G$ is a path of length 4 when $G$ is bipartite and $diam(G)\ge 4$, and (2) is true. Assume that $|V(G)|\ge6$ and that both statements hold for every smaller connected triangle-free graph.

Suppose first that $G$ has a cut vertex $x$. Divide the components of $G-x$ into two nonempty families, and let $G_1$ and $G_2$ be the subgraphs induced by $x$ and the vertices in the two families, respectively. Let $e_i=e(G_i)$ and $d_i=d_{G_i}(x)$. Then both $d_1$ and $d_2$ are positive.

If $G$ is not bipartite, then at least one of $G_1,G_2$ is not bipartite. Assume that $G_1$ is not bipartite. The induction hypothesis yields $f_3(G_1)\ge e_1$. If $G_2$ is not bipartite, then $f_3(G_2)\ge e_2$, and Lemma \ref{lem:triangle-identities}(2) yields $f_3(G)\ge e$. If $G_2$ is bipartite with diameter at least four, then $f_3(G_2)\ge e_2-1$, and $d_1d_2\ge1$ again yields $f_3(G)\ge e$. If $G_2$ is bipartite with diameter at most three, Lemma~\ref{lem:triangle-small-diameter} yields $f_3(G_2)\ge e_2-d_2$, so again by Lemma \ref{lem:triangle-identities}(2), $f_3(G)\ge e_1+e_2-d_2+d_1d_2=e+d_2(d_1-1)\ge e$.

Assume next that $G$ is bipartite and $\diam(G)\ge4$. If both $G_1$ and $G_2$ have diameter at least four, then $f_3(G)\ge(e_1-1)+(e_2-1)+d_1d_2\ge e-1$. If exactly one of $G_1,G_2$, say $G_1$, has diameter at least four, then $G_2$ has diameter at most 3. By Lemma \ref{lem:triangle-small-diameter},   $f_3(G)\ge(e_1-1)+(e_2-d_2)+d_1d_2=e-1+d_2(d_1-1)\ge e-1$. If both $G_1$ and $G_2$ have diameter at most three, then by Lemma \ref{lem:triangle-small-diameter}, $f_3(G)\ge(e_1-d_1)+(e_2-d_2)+d_1d_2=e-1+(d_1-1)(d_2-1)\ge e-1$. 
\vskip 2mm
It remains to consider that $G$ is  $2$-connected. Let $\delta$ be its minimum degree. If $G=K_{\delta,\delta}$, then $G$ is bipartite of diameter two, so neither statement applies. Otherwise, Lemma~\ref{lem:triangle-removal} supplies a vertex $v$ such that $H=G-v$ is connected and $\ell(v)\ge d_G(v)$. Let  $d=d_G(v)$ and $e_H=e(H)$, so $e=e_H+d$.

Suppose that $G$ is not bipartite. If $H$ is not bipartite, the induction hypothesis  yields $f_3(G)=f_3(H)+\ell(v)\ge e_H+d=e$. Assume that $H$ is bipartite, with bipartition $A\cup B$. Since the bipartition does not extend to $G$, the vertex $v$ has neighbors in both parts. Let  $S_A=N_G(v)\cap A$, $S_B=N_G(v)\cap B$, $s=|S_A|$, and $t=|S_B|$. Thus $s,t\ge1$ and $d=s+t$.

Since $G$ is 2-connected and  triangle-free,
the sets $S_A$, $S_B$, $\bigcup_{a\in S_A}N_H(a)\subseteq B$ and $\bigcup_{b\in S_B}N_H(b)\subseteq A$ are disjoint and  nonempty.   Hence $\bigcup_{a\in S_A}N_H(a)\cup \bigcup_{b\in S_B}N_H(b) \subseteq N_2(v)$. Moreover, any pair with one endpoint in $S_A$ and the other in $S_B$ has no common neighbor in the bipartite graph $H$, so $v$ is its only common neighbor in $G$. Therefore $\ell(v)\ge \left|\bigcup_{a\in S_A}N_H(a)\right|+\left|\bigcup_{b\in S_B}N_H(b)\right|+st$.

If $\diam(H)\ge4$, the induction hypothesis  yields $f_3(H)\ge e_H-1$. Since the above four sets  are nonempty, $\ell(v)\ge2+st\ge s+t+1=d+1$, where the second inequality follows from $2+st-(s+t+1)=(s-1)(t-1)\ge0$. Hence $f_3(G)\ge(e_H-1)+(d+1)=e$.

Assume that $\diam(H)\le3$. By \eqref{eq:diameter-count}, $f_3(H)=\binom{|A|}{2}+\binom{|B|}{2}$. Let $\varepsilon=e_H-f_3(H)$. If $\varepsilon\le0$, then $f_3(G)=f_3(H)+\ell(v)\ge e_H+d= e$. Suppose that $\varepsilon\ge1$. Lemma~\ref{lem:triangle-small-diameter} implies $d_H(x)\ge\varepsilon$ for every $x\in V(H)$. Choose $a_0\in S_A$ and $b_0\in S_B$. Their neighborhoods lie in different parts, so $\left|\bigcup_{a\in S_A}N_H(a)\right|+\left|\bigcup_{b\in S_B}N_H(b)\right|\ge d_H(a_0)+d_H(b_0)\ge2\varepsilon$. Thus $\ell(v)\ge2\varepsilon+st\ge s+t+\varepsilon=d+\varepsilon$, because $2\varepsilon+st-(s+t+\varepsilon)=(\varepsilon-1)+(s-1)(t-1)\ge0$. Therefore $f_3(G)\ge(e_H-\varepsilon)+(d+\varepsilon)=e$.

Finally, assume that $G$ is bipartite and $\diam(G)\ge4$. Then all neighbors of $v$ lie in one part of a bipartition of $H$. If $\diam(H)\ge4$, the induction hypothesis yields $f_3(G)\ge(e_H-1)+d=e-1$. Assume that $\diam(H)\le3$, and choose a bipartition $A\cup B$ with $S=N_G(v)\subseteq A$. Let $a=|A|$, $b=|B|$, $|S|=d$, and $R=\bigcup_{u\in S}N_H(u)\subseteq B$ with $r=|R|$.

Since $diam(H)\le 3$, every two vertices of $S$ already have a common neighbor in $H$. So deleting $v$ destroys no triangle-saturating non-edge with both endpoints in $S$. The vertices at distance two from $v$ are exactly those in $R$, and hence $\ell(v)=r$ by the definition of $\ell(v)$.

We claim that  $B\setminus R$ is nonempty. Otherwise, every vertex of $B$ is at distance two from $v$, every vertex of $S$ is adjacent to $v$, and every vertex of $A\setminus S$ is at distance at most three from $v$, since it has distance two in $H$ from any vertex of $S$. Together with $\diam(H)\le3$, this would imply $\diam(G)\le3$, a contradiction.

Since all $d(b-r)$ possible edges between $S$ and $B\setminus R$ are absent,  $e_H\le ab-d(b-r)$. 
Since \(H\) is connected, bipartite, and has diameter at most three,
Lemma~\ref{lem:triangle-small-diameter} gives
\[
f_3(H)=\binom{a}{2}+\binom{b}{2}.
\]
Using \(e=e_H+d\) and \(\ell(v)=r\), we obtain
\begin{align*}
f_3(G)-(e-1)
&=f_3(H)+r-e_H-d+1\\
&\ge
\binom{a}{2}+\binom{b}{2}
+r-ab+d(b-r)-d+1\\
&=
\frac{(a-b)(a-b-1)}{2}
+(d-1)(b-r-1).
\end{align*}
The first term is nonnegative because \(a-b\) is an integer. The second term is also
nonnegative, since \(G\) is \(2\)-connected and hence \(d\ge 2\), and
we have already proved that \(b-r\ge 1\). Consequently,
\( f_3(G)\ge e-1.
\)
This completes the proof. 
\end{proof}

Now we are ready to prove Theorem~\ref{thm:triangle-main} for the connected case.

\begin{theorem}
\label{prop:triangle-connected}
If $G$ is a connected triangle-free graph with $q\ge2$ vertices and $e$ edges, then $f_3(G)\ge\min\{e-1,s(q)\}$.
\end{theorem}

\begin{proof}
If $G$ is not bipartite or $G$ is bipartite with diameter at least four, Lemma~\ref{thm:triangle-structural} yields $f_3(G)\ge e-1$. If $G$ is bipartite with diameter at most three, Lemma~\ref{lem:triangle-small-diameter} yields $f_3(G)\ge s(q)$. 
\end{proof}

Note that
\begin{equation}\label{eq:shift}
s(r+1)=t(r)
\end{equation}
holds for every positive integer $r$.

\begin{lemma}\label{lem:compression}
Let $q\ge2$ and $1\le e\le t(q)$. There is an integer $r$ with $1\le r\le q$ such that $t(r)\ge e$ and $s(r)\le\min\{e-1,s(q)\}$.
\end{lemma}

\begin{proof}
If $s(q)\le e-1$, choose $r=q$. Assume that $e-1<s(q)$. Choose $r$ maximal under the conditions $1\le r<q$ and $s(r)\le e-1$. Such an integer exists because $s(1)=0$. Maximality and $s(q)>e-1$ imply $s(r+1)\ge e$. Hence \eqref{eq:shift} yields $t(r)\ge e$, while $s(r)\le e-1<s(q)$.
\end{proof}

\begin{proof}[\bfseries{Proof of Theorem~\ref{thm:triangle-main}}]
Let $F(n,m)$ denote the right-hand side of \eqref{eq:triangle-formula}. Choose a partition $q_1+\cdots+q_k=n$ attaining $F(n,m)$, and let $H$ be the disjoint union of $T_2(q_1),\ldots,T_2(q_k)$. Then $e(H)=\sum_i t(q_i)\ge m$ and $f_3(H)=\sum_i s(q_i)=F(n,m)$. Delete edges until exactly $m$ edges  remain, and denote the resulting graph by $G$. Since every triangle-saturating non-edge of $G$ is also triangle-saturating in $H$, $f_3(n,m)\le f_3(G)\le F(n,m)$.

For the reverse inequality, let $G$ be a triangle-free graph with $n$ vertices and $m$ edges, and $f_3(G)=f_3(n,m)$. Let $G_1,\ldots,G_\ell$ be its components containing at least one edge. Let $q_j=|V(G_j)|$ and $e_j=e(G_j)$. Since saturating pairs do not join different components, we get  $f_3(G)=\sum_j f_3(G_j)$. By Theorem~\ref{prop:triangle-connected},  $f_3(G_j)\ge\min\{e_j-1,s(q_j)\}$.

By Mantel's theorem, $e_j\le t(q_j)$. Lemma~\ref{lem:compression} supplies an integer $r_j\le q_j$ with $t(r_j)\ge e_j$ and $s(r_j)\le\min\{e_j-1,s(q_j)\}\le f_3(G_j)$. Note that $n\ge q_1+\cdots +q_{\ell}$. Replace each part of  $q_j$ by one part of size $r_j$ and $q_j-r_j$ parts of size one. Together with each  part of size one for every isolated vertex of $G$, these parts form a partition of $n$, their total $t$-value is at least $\sum_j e_j=m$, and their total $s$-value is at most $\sum_j f_3(G_j)=f_3(G)$. Hence $F(n,m)\le f_3(G)= f_3(n,m)$, and the theorem follows.
\end{proof}

\section{Concluding remarks}

For $p\ge3$, the upper-bound construction is a blow-up of the graph underlying the constructions in~\cite{BaloghLiu,HMMY}. The proof establishes its asymptotic optimality, but it does not characterize all extremal graphs. Two natural next questions are to obtain a stability theorem for near-extremal graphs and to determine the linear term uniformly throughout the interval.

For $p=2$, Theorem~\ref{thm:triangle-main} gives an exact finite
optimization formula, and Theorem~\ref{cor:triangle-dp} provides a
dynamic program for its evaluation. It remains natural to ask
whether this optimization admits a simpler closed form on substantial
subintervals of $0\le m\le t_2(n)$.

\section*{Acknowledgements}

Jiabao Yang and Ruilin Zheng are supported by the National Key R\&D Program of China under Grant No.~2024YFA1013900, by the National Natural Science Foundation of China under Grant No.~12471327, and, for Yang, additionally by the China Postdoctoral Science Foundation under Grant No.~2026M793375. Xiaolin Wang is supported by the National Natural Science Foundation of China under Grant No.~12401447.

\section*{Declaration on the use of AI}
During the preparation of this work, the authors used OpenAI's ChatGPT (GPT-5.6 Pro) to assist with language editing, organization, and preliminary checks of algebraic and logical consistency. After using this tool, the authors reviewed and edited the content as needed, independently verified all mathematical arguments, and take full responsibility for the content of the article.

\end{document}